\documentclass[11pt]{amsart}
\usepackage{a4wide}
\usepackage{version}
\usepackage{amsmath}
\usepackage{amsfonts, amssymb, amsthm}
\usepackage{pstricks}
\usepackage{color,mathrsfs,epsfig}
\usepackage{epsfig}
\usepackage{color}

\DeclareMathAlphabet{\mathpzc}{OT1}{pzc}{m}{it}
\newcounter{comptage}[part]
\newtheorem{lem}[comptage]{Lemma}
\newtheorem{theo}[comptage]{Theorem}
\newtheorem{defin}[comptage]{Definition}
\newtheorem{cor}[comptage]{Corollary}
\newtheorem{prop}[comptage]{Proposition}

\newtheorem{preremark}[comptage]{Remark}
\newenvironment{remark}{\begin{preremark}\rm}{\medskip \end{preremark}}

\numberwithin{equation}{section}

\newcommand{\Hh}{\mathcal H}

\newcommand{\R}{\mathbb R}
\newcommand{\N}{\mathbb N}

\newcommand{\diam}{{\rm Diam}}

\newcommand{\meanbar}[1]{%
\setbox0 = \hbox{$#1 \int$}
\hbox to 0pt{%
\thinspace
\hskip 0.1\wd0
\raise 0.5\ht0
\hbox{%
\lower 0.5\dp0
\hbox{\rule{0.8\wd0}{2\linethickness}}
}%
\hss
}%
}

\author{Antoine Lemenant, Emmanouil Milakis and Laura V. Spinolo}

\title[On the extension property of Reifenberg-flat domains]{On the extension property of Reifenberg-flat domains.}

\begin{document}

\begin{abstract} We provide a detailed proof of the fact that any domain which is sufficiently flat in the sense of Reifenberg is also Jones-flat, and hence it is an extension domain. We discuss various applications of this property, in particular we obtain $L^\infty$ estimates for the eigenfunctions of the Laplace operator with Neumann boundary conditions. We also compare different ways of measuring the ``distance" between two sufficiently close Reifenberg-flat domains. These results are pivotal to the quantitative stability analysis of the spectrum of the Neumann Laplacian performed in~\cite{lms2}.

\end{abstract}

\maketitle



\section{Introduction}
The main goal of the present paper is establishing extension and geometric properties for a class of domains whose boundaries satisfy a fairly weak regularity requirement introduced by Reifenberg~\cite{r}. 
In particular, we show that any domain that is sufficiently flat in the sense of Reifenberg enjoys the so-called extension property and we discuss applications that are relevant for the analysis of PDEs defined in these domains. We also compare different ways of measuring the ``distance" between two sufficiently close Reifenberg-flat domains $X$ and $Y$, in particular we discuss the relations between the Hausdorff distances $d_H(X, Y)$, $d_H(\R^N \setminus X, \R^N \setminus Y)$ and $d_H(\partial X, \partial Y)$ and the measure of the symmetric difference $|X \triangle Y|$.

Although we are confident our results can find different applications, our original motivation was the quantitative stability analysis of the spectrum of the Laplace operator with Neumann boundary conditions defined in 
Reifenberg-flat domains, see~\cite{lms2}.  

The notion of Reifenberg-flat sets was first introduced in 1960 by Reifenberg~\cite{r} when he was working on the Plateau problem, and has since then played an important role in the study of minimal surfaces. More recently, the works by David \cite{d1,d2} about the regularity for 2-dimensional minimal sets in $\R^N$ rely on the Reifenberg parametrization and the specific $3$-dimensional results by David, De Pauw and Toro \cite{ddpt}. Also, Reifenberg-flat set are relevant in the study of the harmonic measure (see Kenig and Toro~\cite{hm3,hm2,hm1} and Toro~\cite{t,hm4}) and of the regularity for free boundary problems, like the minimization of the Mumford-Shah functional (see \cite{l1,l2}). Elliptic and parabolic equations defined in Reifenberg-flat domains have been recently investigated by Byun, Wang and Zhou~\cite{w1,w2,w3}, by Lemenant, Milakis and Spinolo and by Milakis and Toro~\cite{lm2,lm,lms2, mt}. Finally, we mention that Reifenberg-flat domains are in particular NTA domains in the sense of Jerison and Kenig 
\cite{jk}. 

We now provide the precise definition. We denote by $d_H$ the classical Hausdorff distance between two sets $X$ and $Y$,
\begin{equation}
    d_H( X , Y ) : = \max \big\{ \sup_{x \in  X } d(x, Y),
   \sup_{y \in Y} d(y, X)  \big\}.
\end{equation} 
\begin{defin}\label{defreif} Let $\varepsilon, r_0$ be two real numbers satisfying $0 < \varepsilon<1/2$ and $r_0 >0$. An $(\varepsilon,r_0)$-Reifenberg-flat domain $\Omega \subseteq \R^N$ is a nonempty open set satisfying the following two conditions:
\begin{itemize}
\item[$i)$] for every $x \in \partial \Omega$ and for every $r\leq r_0$,  there is a hyperplane $P(x,r)$ containing $x$ which satisfies
\begin{eqnarray}
\frac{1}{r}d_H\big( \partial \Omega \cap B(x,r), P(x,r)\cap B(x,r) \big) \leq \varepsilon. \label{reif}
\end{eqnarray}
\item[$ii)$]For every $x \in \partial \Omega$,  one of the connected component of  
$$B(x,r_0)\cap \big\{x : \;  dist(x,P(x,r_0))\geq 2\varepsilon r_0\big\}$$ 
is contained in   $\Omega$ and the other one is contained in $\R^N \setminus \Omega$.
\end{itemize}
\end{defin}
Condition $i)$ states that the boundary of $\Omega$ is an $(\varepsilon,r_0)$-Reifenberg-flat set. A Reifenberg-flat set enjoys local separability properties (see e.g. Theorem 4.1. in \cite{lihewang}), however we observe that condition $ii)$ in the definition  is not in general implied by condition $i)$, as the example of $\Omega = \R^N \setminus \partial B(0,1)$ shows (here $\partial B(0, 1)$ denotes the boundary of the unit ball). However, a consequence of the analysis in David~ \cite{d0} is that $i)$ implies $ii)$ under some further topological assumption, for instance the implication holds if $\Omega$ and $\partial \Omega$ are both connected. Note furthermore that a straightforward consequence of the definition is that, if $\varepsilon_1 < \varepsilon_2$, then any $(\varepsilon_1,r_0)$-Reifenberg-flat domain is also an $(\varepsilon_2,r_0)$-Reifenberg-flat domain. Finally, note that we only impose the separability requirement $ii)$  at scale $r_0$ but it simply follows from the definition that it also holds at any scale $r\leq r_0$ (see \cite[Proposition 2.2]{hm3} or Lemma~\ref{l:ii} below).

In \cite{r} Reifenberg proved the so-called topological disk theorem which states that, provided $\varepsilon$ is small enough,  any $(\varepsilon, r_0)$-Reifenberg-flat set in the unit $N$-ball is the bi-H\"olderian image of an $(N-1)$-dimensional disk. Also, any Lipschitz domain with sufficiently small Lipschitz constant is Reifenberg-flat for a suitable choice of the regularity parameter $\varepsilon$ (the choice depends on the Lipschitz constant). On the other hand, the ``flat'' Koch snowflake with sufficiently small angle is Reifenberg-flat (see Toro~\cite{t}) and hence it is an example of a Reifenberg-flat set which is \emph{not} Lipschitz, and with Hausdorff dimension greater than $N-1$.

The main goal of this paper is providing a complete and detailed proof of the fact that Reifenberg-flat domains are extension domains. This fact is relevant for the study of elliptic problems and was already known and used in the literature (see e.g. the introduction of \cite{w1}). However, to the best of our knowledge, an explicit proof was so far missing. We recall that the so called \textit{extension problem} can be formulated as follows: given an open set $\Omega$, we denote by $W^{1,p}$ the classical Sobolev space and we wonder whether or not one can define a bounded linear operator (the so-called extension operator) 
$$ 
    E:W^{1,p}(\Omega)\rightarrow W^{1,p}(\mathbb{R}^N)
$$  
such that $E(u) \equiv u$ on $\Omega$. If $\partial\Omega$ is Lipschitz, Calderon \cite{Calderon} established the existence of an extension operator in the case when $1<p<\infty$, while Stein \cite{Stein} considered the cases $p=1,\infty$. Jones \cite{J} proved the existence of extension operators for a new class of domains, the so-called $(\varepsilon,\delta)$-Jones flat domains (the precise definition is recalled in Section~\ref{topos}). In the present work we prove that sufficiently flat Reifenberg domains are indeed Jones flat domains. Our main result concerning the extension problem is as follows.

\begin{theo}\label{jonesrei} 
Any   $(1/600 , r_0)$-Reifenberg flat domain is a $(1/450 , r_0/7)$-Jones flat domain. 
\end{theo}

As direct consequence of Theorem~\ref{jonesrei} we get that one can define extension operators for $(1/600 , r_0)$-Reifenberg flat domains (see Corollary \ref{ext} for a precise statement). Some relevant features of this result are the following: first, we provide an explicit and universal threshold 
on the coefficient $\varepsilon$ for the extension property to hold (namely, $\varepsilon \leq 1/600$). Second, $1/600$ is fairly big compared to the usual threshold needed to apply Reifenberg's topological disk theorem (for e.g. the threshold is $10^{-15}$ in \cite{ddpt}, see also \cite{lihewang2} for an interesting alternative proof). 

As a consequence of the extension extension property, we obtain that the classical  Rellich-Kondrachov Theorem applies to Reifenberg-flat domains (see Proposition~\ref{embedding}), that the Neumann Laplacian has a discrete spectrum and that the eigenfunctions are bounded (see Proposition~\ref{uestim}).  Also, by combining Theorem~\ref{jonesrei} with the works by Chua \cite{Chua-indiana,Chua-illinois,Chua-Canad} and Christ \cite{Christ} we get that one can define extension operators for weighted Sobolev spaces and Sobolev spaces of fractional order (see Remark~\ref{chuachrist} in the present paper).

We conclude the paper by establishing results unrelated to the extension problem, namely we study the relation between different ways of measuring the ``distance" between sets of $\R^N$. In particular, for two general open sets $X$ and $Y$, neither the Hausdorff distance $d_{H}(X,Y)$ nor the Hausdorff distance between the complements $d_H(\R^N \setminus X, \R^N \setminus Y)$ is, in general, controlled by the Lebesgue measure of the symmetric difference $|X \triangle Y|$. However, we prove that they are indeed controlled provided that $X$, $Y$ are Reifenberg flat and close enough, in a suitable sense. This result will be as well applied in~\cite{lms2} to the stability analysis of the spectrum of the Laplace operator with Neumann boundary conditions. 

The paper is organized as follows: in Section~\ref{topos} we prove that sufficiently flat Reifenberg domains are Jones-flat, in Section~\ref{appl} we show that if these domains are also connected, then they enjoy the extension property. In Section~\ref{appl} we also discuss some applications of the extension property. In Section~\ref{connected} we investigate how to handle domains that are not connected and finally in Section~\ref{s:hausdorff} we investigate the relation between different ways of measuring the ``distance" between Reifenberg-flat domains. 
\subsection{Notations}
We denote by $C(a_1, \dots, a_h)$ a constant only depending on the variables $a_1, \dots, a_h$. Its precise value can vary from line to line. 
Also, we use the following notations:  \\
$\Hh^N$ : the $N$-dimensional Hausdorff measure.\\
$\omega_N$ : the Lebesgue measure of the unit ball in $\R^N$.\\
$|A|$ : the Lebesgue measure of the Borel set $A  \subseteq \R^N$. \\
$A^c$: the complement of the set $A$, $A^c : = \R^N \setminus A.$ \\
$\bar A$: the closure of the set $A$. \\
$W^{1,p}(\Omega)$ : the Sobolev space of $L^p$ functions whose derivatives are in $L^p$.\\
$\langle x, y \rangle$: the standard scalar product between the vectors $x, y \in \R^N$. \\
$|x|$: the norm of the vector $x \in \R^N$. \\
$d(x, y)$: the distance from the point $x$ to the point $y$, $d(x, y) = |x -y|$. \\
$d(x, A)$: the distance from the point $x$ to the set $A$. \\
$d_H(A,B)$ : the Hausdorff distance from the set $A$ to the set $B$.\\
%$rad (A):$ the radius of the nonempty set $A \subseteq \R^N$,
%$$
 %   rad (A) : = \inf_{x \in A} \sup_{y \in A} d(x, y).
%$$
$[x, y]$: the segment joining the points $x, y \in \R^N$. \\
$B(x, r)$: the open ball of radius $r$ centered at $x$.  \\
$\overline{B}(x, r)$: the closed ball of radius $r$ centered at $x$. 

\section{Reifenberg-flat and Jones domains}
\label{topos}
In this section we show that any sufficiently flat Reifenberg domain is Jones-flat, in the sense of~\cite{J}. The extension property follows then as a corollary of the analysis in~\cite{J}.

First, we provide the precise definition of Jones-flatness. 
\begin{defin} An open and bounded set $\Omega$ is a $(\delta, R_0)$-Jones-flat domain if for any $x,y \in \Omega$ such that $d(x,y)\leq R_0$ there is a rectifiable curve $\gamma$ which connects $x$ and $y$ and satisfies 
\begin{equation}
\label{e:lenght}
{\Hh^1(\gamma)\leq \delta^{-1}d(x,y)}
\end{equation} 
and
\begin{eqnarray}
d(z,\Omega^c)\geq \delta \, \frac{d(z,x)d(z,y)}{d(x,y)},\quad \text{ for all } z\in \gamma. \label{jonesprop3}
\end{eqnarray}
\end{defin}
To investigate the relation between Jones flatness and Reifenberg flatness we need two preliminary lemmas. 
\begin{lem} \label{angleLemma}Let $\Omega \subseteq \R^N$ be an $(\varepsilon, r_0)-$Reifenberg flat domain. Given $x \in \partial \Omega$ and $r \leq r_0$, we term $\nu_r$ the unit normal vector to the hyperplane $P(x,r)$ provided by the definition of Reifenberg-flatness. Given $M \ge 1$,  for every $r\leq r_0/M$ we have
\begin{equation}
\label{e:normal}
 |\langle \nu_r, \nu_{M r}\rangle| \geq 1 - (M+1)\varepsilon.
 \end{equation}
\end{lem}
\begin{proof} We assume with no loss of generality that $x$ is the origin. For simplicity, in the proof we denote by $B_r$ the ball $B(0,r)$ and by $P_r$ the hyperplane $P(0,r)$. From the definition of Reifenberg flatness we infer that
\begin{eqnarray}
d_H(P_{M r}\cap B_r , P_r \cap B_r ) & \leq & d_H(P_{M r}\cap B_r , \partial \Omega \cap B_r ) +  d_H(\partial \Omega \cap B_r , P_r \cap B_r ) \notag \\
&\leq & M r \varepsilon + r\varepsilon \leq (M+1)r \varepsilon. \notag
\end{eqnarray}
Since $P_{Mr}$ and $P_r$ are linear spaces we deduce that
\begin{eqnarray}
d_H(P_{M r}\cap B_1 , P_r \cap B_1 )\leq (M+1)\varepsilon. \label{distP}
\end{eqnarray}
We term $\pi_r$ and $\pi_{Mr}$ the orthogonal  projections onto $P_r$ and $P_{Mr}$, respectively, and we fix an arbitrary point $y \in P_r\cap B_1$. Inequality~\eqref{distP} states that there is $z\in \bar P_{Mr}\cap \bar B_1$ satisfying 
$$ d(z, y)\leq (M+1)\varepsilon.$$
In particular, since $1= |\nu_{Mr} |=\inf_{z \in P_{Mr}} d( \nu_{Mr}, z)$, we get
$$d(\nu_{Mr}, y ) \geq d(\nu_{Mr}, z )- d(z, y ) \geq 1-(M+1)\varepsilon.$$
By taking the infimum for $y \in P_r\cap B_1$ we obtain
$$ |\nu_{Mr}- \pi_{r}(\nu_{Mr}) |\geq  1-(M+1)\varepsilon,$$
and the proof is concluded by recalling that $ |\langle \nu_{M r}, \nu_r \rangle| = d(\nu_{Mr}, \pi_{r}(\nu_{Mr}) )$.
\end{proof}
The following lemma discuss an observation due to Kenig and Toro~\cite[Proposition 2.2]{hm3}. Note that the difference between Lemma~\ref{l:ii} and part $ii)$ in the definition of Reifenberg flatness is that in $ii)$ we only require the separation property at scale $r_0$.
\begin{lem}
\label{l:ii} Let $\Omega \subseteq \R^N$ be an $(\varepsilon, r_0)-$Reifenberg flat domain. For every $x \in \partial \Omega$ and $r \in ]0, r_0]$, one of the connected components of  
$$B(x,r)\cap \big\{x : \;  dist(x,P(x,r))\geq 2\varepsilon r \big\}$$ 
is contained in  $\Omega$ and the other one is contained in $\R^N \setminus \Omega$. Here $P(x, r)$ is the same hyperplane as in part $i)$ of the definition of Reifenberg-flatness.
\end{lem}
\begin{proof}
We fix $\rho \in ]0, r_0]$ and we assume that the separation property holds at scale $\rho$, namely that  one of the connected components of  
$$B(x, \rho)\cap \big\{x : \;  dist(x,P(x,\rho))\geq 2\varepsilon \rho \big\}$$ 
is contained in  $\Omega$ and the other one is contained in $\R^N \setminus \Omega$. We now show that the same separation property holds at scale $r$ for every $r \in ]\rho/M, \rho]$ provided that $M \leq (1- \varepsilon) / 3 \varepsilon.$ By iteration this implies that the separation property holds at any scale $r \in ]0, r_0].$

Let us fix $r \in ]\rho/M, \rho]$ and denote by $B^+ (x, r) $ one of the connected components of 
$$
  B(x, r)\cap \big\{x : \;  dist(x,P(x,r)) \geq 2\varepsilon r \big\}$$
and by $B^-(x, r)$ the other one.  Also, we term $Y^+$ and $Y^-$ the points of intersection of the line passing through $x$ and perpendicular to 
$P(x, r)$ with the boundary of the ball $B(x, r)$. 

By recalling~\eqref{e:normal} and the inequality $r \ge \rho / M$, we get that the distance of $Y^{\pm}$ from the hyperpane $P(x, \rho)$ satisfies the following inequality:
$$
    d( Y^{\pm}, P(x, \rho)) \ge  r  
    | \langle \nu_\rho, \nu_{\rho/M} \rangle| \geq r \big[ 1 -  
    (M+1)\varepsilon \big] \ge \rho \frac{ 1 -  
    (M+1)\varepsilon }{M}. 
$$
Since by assumption $M \leq (1-   \varepsilon) / 3 \varepsilon$, this implies that 
$d( Y^{\pm}, P(x, \rho)) \ge  2 \varepsilon \rho$ and hence that one among $Y^+$ and $Y^-$ belongs to $B^+(x, \rho)$ and the other one to $B^-(x, \rho)$. Since by assumption the separation property holds at scale $\rho$, this implies that one of them belongs to $\Omega$ and the other one to $\Omega^c$.

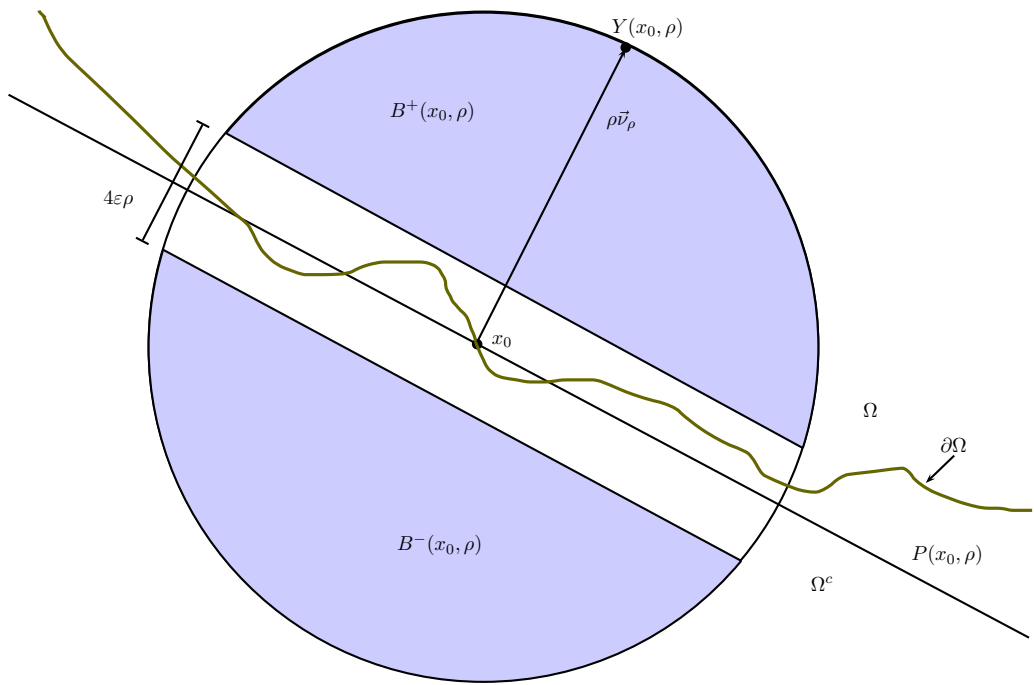
\begin{figure}[h]
\begin{center}
% Generated with LaTeXDraw 2.0.8
% Fri Jun 01 18:15:39 CEST 2012
% \usepackage[usenames,dvipsnames]{pstricks}
% \usepackage{epsfig}
% \usepackage{pst-grad} % For gradients
% \usepackage{pst-plot} % For axes
\begin{center}
\scalebox{0.7} % Change this value to rescale the drawing.
{
\begin{pspicture}(0,-6.4)(19.47,6.41)
\definecolor{color613b}{rgb}{0.8,0.8,1.0}
\definecolor{color554}{rgb}{0.4,0.4,0.0}
\pscircle[linewidth=0.04,dimen=outer](9.02,0.0){6.38}
\psline[linewidth=0.04cm](0.0,4.78)(19.38,-5.54)
\psdots[dotsize=0.2](8.9,0.04)
\usefont{T1}{ptm}{m}{n}
\rput(9.37,0.105){$x_0$}
\usefont{T1}{ptm}{m}{n}
\rput(17.83,-3.975){$P(x_0,\rho)$}
\psarc[linewidth=0.04,fillstyle=solid,fillcolor=color613b](9.02,-0.02){6.36}{342.47443}{140.24182}
\psline[linewidth=0.04](15.084785,-1.9351954)(4.130747,4.04753)
\psarc[linewidth=0.04,fillstyle=solid,fillcolor=color613b](9.02,-0.02){6.36}{163.05739}{320.33215}
\psline[linewidth=0.04](2.9360423,1.8333911)(13.915661,-4.0798163)
\pscustom[linewidth=0.06,linecolor=color554]
{
\newpath
\moveto(0.66,6.28)
\lineto(0.75,6.08)
\curveto(0.795,5.98)(0.89,5.805)(0.94,5.73)
\curveto(0.99,5.655)(1.085,5.535)(1.13,5.49)
\curveto(1.175,5.445)(1.44,5.195)(1.66,4.99)
\curveto(1.88,4.785)(2.335,4.345)(2.57,4.11)
\curveto(2.805,3.875)(3.08,3.605)(3.12,3.57)
\curveto(3.16,3.535)(3.28,3.43)(3.36,3.36)
\curveto(3.44,3.29)(3.605,3.145)(3.69,3.07)
\curveto(3.775,2.995)(3.915,2.87)(3.97,2.82)
\curveto(4.025,2.77)(4.145,2.66)(4.21,2.6)
\curveto(4.275,2.54)(4.38,2.45)(4.42,2.42)
\curveto(4.46,2.39)(4.525,2.33)(4.55,2.3)
\curveto(4.575,2.27)(4.64,2.15)(4.68,2.06)
\curveto(4.72,1.97)(4.795,1.84)(4.83,1.8)
\curveto(4.865,1.76)(4.95,1.675)(5.0,1.63)
\curveto(5.05,1.585)(5.165,1.515)(5.23,1.49)
\curveto(5.295,1.465)(5.415,1.42)(5.47,1.4)
\curveto(5.525,1.38)(5.605,1.36)(5.63,1.36)
\curveto(5.655,1.36)(5.73,1.36)(5.78,1.36)
\curveto(5.83,1.36)(5.94,1.36)(6.0,1.36)
\curveto(6.06,1.36)(6.205,1.365)(6.29,1.37)
\curveto(6.375,1.375)(6.495,1.4)(6.53,1.42)
\curveto(6.565,1.44)(6.635,1.475)(6.67,1.49)
\curveto(6.705,1.505)(6.775,1.53)(6.81,1.54)
\curveto(6.845,1.55)(6.905,1.565)(6.93,1.57)
\curveto(6.955,1.575)(7.005,1.585)(7.03,1.59)
\curveto(7.055,1.595)(7.115,1.6)(7.15,1.6)
\curveto(7.185,1.6)(7.25,1.6)(7.28,1.6)
\curveto(7.31,1.6)(7.37,1.6)(7.4,1.6)
\curveto(7.43,1.6)(7.495,1.6)(7.53,1.6)
\curveto(7.565,1.6)(7.64,1.6)(7.68,1.6)
\curveto(7.72,1.6)(7.785,1.6)(7.81,1.6)
\curveto(7.835,1.6)(7.89,1.595)(7.92,1.59)
\curveto(7.95,1.585)(8.025,1.56)(8.07,1.54)
\curveto(8.115,1.52)(8.175,1.47)(8.19,1.44)
\curveto(8.205,1.41)(8.24,1.35)(8.26,1.32)
\curveto(8.28,1.29)(8.3,1.235)(8.3,1.21)
\curveto(8.3,1.185)(8.305,1.135)(8.31,1.11)
\curveto(8.315,1.085)(8.34,1.03)(8.36,1.0)
\curveto(8.38,0.97)(8.405,0.92)(8.41,0.9)
\curveto(8.415,0.88)(8.44,0.84)(8.46,0.82)
\curveto(8.48,0.8)(8.52,0.755)(8.54,0.73)
\curveto(8.56,0.705)(8.605,0.66)(8.63,0.64)
\curveto(8.655,0.62)(8.695,0.57)(8.71,0.54)
\curveto(8.725,0.51)(8.75,0.445)(8.76,0.41)
\curveto(8.77,0.375)(8.79,0.32)(8.8,0.3)
\curveto(8.81,0.28)(8.83,0.235)(8.84,0.21)
\curveto(8.85,0.185)(8.865,0.13)(8.87,0.1)
\curveto(8.875,0.07)(8.89,0.02)(8.9,0.0)
\curveto(8.91,-0.02)(8.93,-0.065)(8.94,-0.09)
\curveto(8.95,-0.115)(8.975,-0.17)(8.99,-0.2)
\curveto(9.005,-0.23)(9.035,-0.295)(9.05,-0.33)
\curveto(9.065,-0.365)(9.105,-0.425)(9.13,-0.45)
\curveto(9.155,-0.475)(9.21,-0.52)(9.24,-0.54)
\curveto(9.27,-0.56)(9.35,-0.59)(9.4,-0.6)
\curveto(9.45,-0.61)(9.54,-0.625)(9.58,-0.63)
\curveto(9.62,-0.635)(9.705,-0.65)(9.75,-0.66)
\curveto(9.795,-0.67)(9.875,-0.68)(9.91,-0.68)
\curveto(9.945,-0.68)(10.015,-0.68)(10.05,-0.68)
\curveto(10.085,-0.68)(10.15,-0.68)(10.18,-0.68)
\curveto(10.21,-0.68)(10.28,-0.675)(10.32,-0.67)
\curveto(10.36,-0.665)(10.425,-0.655)(10.45,-0.65)
\curveto(10.475,-0.645)(10.565,-0.64)(10.63,-0.64)
\curveto(10.695,-0.64)(10.825,-0.64)(10.89,-0.64)
\curveto(10.955,-0.64)(11.055,-0.64)(11.09,-0.64)
\curveto(11.125,-0.64)(11.205,-0.655)(11.25,-0.67)
\curveto(11.295,-0.685)(11.5,-0.77)(11.66,-0.84)
\curveto(11.82,-0.91)(12.145,-1.015)(12.31,-1.05)
\curveto(12.475,-1.085)(12.69,-1.18)(12.74,-1.24)
\curveto(12.79,-1.3)(13.01,-1.475)(13.18,-1.59)
\curveto(13.35,-1.705)(13.635,-1.86)(13.75,-1.9)
\curveto(13.865,-1.94)(14.015,-2.01)(14.05,-2.04)
\curveto(14.085,-2.07)(14.15,-2.155)(14.18,-2.21)
\curveto(14.21,-2.265)(14.26,-2.35)(14.28,-2.38)
\curveto(14.3,-2.41)(14.34,-2.45)(14.36,-2.46)
\curveto(14.38,-2.47)(14.475,-2.515)(14.55,-2.55)
\curveto(14.625,-2.585)(14.825,-2.665)(14.95,-2.71)
\curveto(15.075,-2.755)(15.28,-2.785)(15.36,-2.77)
\curveto(15.44,-2.755)(15.59,-2.67)(15.66,-2.6)
\curveto(15.73,-2.53)(15.87,-2.45)(15.94,-2.44)
\curveto(16.01,-2.43)(16.175,-2.41)(16.27,-2.4)
\curveto(16.365,-2.39)(16.565,-2.365)(16.67,-2.35)
\curveto(16.775,-2.335)(16.93,-2.32)(16.98,-2.32)
\curveto(17.03,-2.32)(17.125,-2.4)(17.17,-2.48)
\curveto(17.215,-2.56)(17.425,-2.71)(17.59,-2.78)
\curveto(17.755,-2.85)(18.02,-2.95)(18.12,-2.98)
\curveto(18.22,-3.01)(18.375,-3.05)(18.43,-3.06)
\curveto(18.485,-3.07)(18.575,-3.08)(18.61,-3.08)
\curveto(18.645,-3.08)(18.73,-3.08)(18.78,-3.08)
\curveto(18.83,-3.08)(18.915,-3.09)(18.95,-3.1)
\curveto(18.985,-3.11)(19.045,-3.12)(19.07,-3.12)
\curveto(19.095,-3.12)(19.145,-3.12)(19.17,-3.12)
\curveto(19.195,-3.12)(19.245,-3.12)(19.27,-3.12)
\curveto(19.295,-3.12)(19.345,-3.12)(19.37,-3.12)
\curveto(19.395,-3.12)(19.425,-3.12)(19.44,-3.12)
}
\pscustom[linewidth=0.06,linecolor=color554]
{
\newpath
\moveto(0.7,6.24)
\lineto(0.67,6.27)
\curveto(0.655,6.285)(0.625,6.315)(0.61,6.33)
\curveto(0.595,6.345)(0.58,6.365)(0.58,6.38)
}
\usefont{T1}{ptm}{m}{n}
\rput(17.95,-1.875){$\partial \Omega$}
\psline[linewidth=0.04cm,arrowsize=0.05291667cm 2.0,arrowlength=1.4,arrowinset=0.4]{<-}(17.42,-2.6)(17.96,-2.08)
\usefont{T1}{ptm}{m}{n}
\rput(16.37,-1.215){$\Omega$}
\usefont{T1}{ptm}{m}{n}
\rput(15.45,-4.515){$\Omega^c$}
\psline[linewidth=0.04cm,tbarsize=0.07055555cm 5.0]{|*-|}(2.54,2.0)(3.7,4.24)
\psline[linewidth=0.04cm,arrowsize=0.05291667cm 2.0,arrowlength=1.4,arrowinset=0.4]{<-}(11.72,5.66)(8.9,0.08)
\psdots[dotsize=0.2](11.72,5.68)
\usefont{T1}{ptm}{m}{n}
\rput(12.14,6.045){$Y(x_0,\rho)$}
\usefont{T1}{ptm}{m}{n}
\rput(2.09,2.785){$4\varepsilon \rho$}
\usefont{T1}{ptm}{m}{n}
\rput(11.64,4.285){$\rho \vec \nu_\rho$}
\usefont{T1}{ptm}{m}{n}
\rput(8.05,4.465){$B^+(x_0,\rho)$}
\usefont{T1}{ptm}{m}{n}
\rput(8.19,-3.775){$B^-(x_0,\rho)$}
\end{pspicture} 
}
\end{center}
\end{center}
\caption{notations for the proof of Theorem \ref{jonesrei}}
\label{fig1}
\end{figure}

To conclude, note that part $i)$ in the definition of Reifenberg flatness implies that 
$$
    B^{\pm} (x, r) \cap \partial \Omega = \emptyset 
$$ 
and hence both $B^+(x, r)$ and $B^-(x, r)$ are entirely contained in either $\Omega$ or $\Omega^c$. By recalling that one among $Y^+$ and $Y^-$ belongs to $\Omega$ and the other one to $\Omega^c$, we conclude 
the proof of the lemma. 
\end{proof}
We are now ready to establish the main result of this section, namely Theorem \ref{jonesrei}.

\begin{proof}[Proof of Theorem \ref{jonesrei}]
We assume $\varepsilon \leq 1/600$, we fix an $(\varepsilon, r_0)$-Reifenberg flat domain $\Omega \subseteq \R^N$  and we proceed according to the following steps. \\
{\sc $\diamond$ Step 1.} We first introduce some notations (see Figure~\ref{fig1} for a representation).

%\begin{figure}[h]
%\begin{center}
%\input{pictureProp12Z.tex}
%\end{center}
%\caption{notations for the proof of Theorem \ref{jonesrei}}
%\label{fig1}
%\end{figure}

For any $x_0 \in \partial \Omega$ and $\rho \leq r_0$, we denote as usual by $P(x_0,\rho)$ the hyperplane provided by the definition of Reifenberg flatness, and by $\vec \nu_\rho$ its normal. By Lemma~\ref{l:ii}, we can choose the orientation of $\vec \nu_\rho$ in such a way that  
$$
    B^+ (x_0, \rho) : = 
    \left\{ z +t  \vec \nu_\rho: \;   z \in P(x_0, \rho), \, 
    t \ge 2 \varepsilon  \rho \right\} \cap B(x_0, \rho ) \subseteq \Omega
$$
and 
$$
      B^- (x_0, \rho) : = 
    \left\{ z -t  \vec \nu_\rho: \;   z \in P(x_0, \rho), \, 
    t \ge 2 \varepsilon  \rho \right\} \cap B(x_0, \rho ) \subseteq \Omega^c.
$$
Also, we define the hyperplanes $ P^+ (x_0, \rho)$ and $ P^- (x_0, \rho)$ by setting
$$
     P^+ (x_0, \rho) : = 
    \left\{ z +2 \varepsilon \rho  \vec \nu_\rho: \;   z \in P(x_0, \rho) 
 \right\}
$$
and 
$$
     P^- (x_0, \rho) : = 
    \left\{ z -2 \varepsilon \rho  \vec \nu_\rho: \;   z \in P(x_0, \rho) 
 \right\}
$$
and we denote by $Y(x_0, \rho)$ the point 
$$
    Y(x_0, \rho) : = x_0 + \rho \vec \nu_\rho. 
$$

Finally, for any $x \in \Omega$, we denote by $x_0 \in \partial \Omega$ the point such that $d(x, \Omega^c)=d(x,x_0)$ (if there is more than one such $x_0$, we arbitrarily fix one).

{\sc $\diamond$ Step 2.} We provide a preliminary construction: more precisely, given 
\begin{itemize}
\item $x \in \Omega$ such that
$d(x, \Omega^c) \leq 2r_0/7$ and 
\item $r$ satisfying $d(x, \Omega^c) /2  \leq r \leq  r_0/7$, 
\end{itemize}
the curve $\gamma_{x,r}$ is defined as follows. 
\begin{itemize}
\item[(I)] If $d(x, \Omega^c)/2 \leq r  \leq 2 d (x, \Omega^c)$, then $\gamma_{x,r}$ is simply the segment $[x, Y(x_0, r)]$. 
\item[(II)] If $2 d(x, \Omega^c) < r \leq r_0/7$, we denote by $k_0 \ge 1$ the biggest natural number $k$ satisfying 
$ {2^{-k} r \ge d(x, \Omega^c)}$ and we set 
$$
     \gamma_{x,r} := [x,Y(x_0, 2^{-k_0}r)] \cup 
     \bigcup_{k =0}^{k_0-1} [Y(x_0,2^{-k}r), Y(x_0,2^{-(k+1)}r)].
$$
\end{itemize}  
{\sc $\diamond$ Step 3.} We prove that in both cases (I) and (II) we have 
\begin{equation}
 \label{jonesprop2}
            \Hh^1(\gamma_{x,r})  \leq 4 r. 
\end{equation}
To handle case (I) we just observe that, since by assumption $d (x, \Omega^c) = d(x, x_0) \leq 2r$, then, 
by recalling $d(x_0, Y(x_0, r))=r$, property~\eqref{jonesprop2} follows.

To handle case (II), we first observe that, since $d(x, x_0) \leq 2^{-k_0}r$, then both $x$ and $Y(x_0, 2^{-k_0}r)$ belong to the closure of $B(x_0, 2^{-k_0}r)$. Also, by construction both $Y(x_0,2^{-k}r)$ and $Y(x_0,2^{-(k+1)}r)$ belong to the closure of 
$B(x_0, 2^{-k}r)$ and by combining these observations we conclude that 
\begin{eqnarray}
 \Hh^1(\gamma_{x,r}) &\leq& d(x, Y(x_0, 2^{-k_0}r))+ \sum_{k=0}^{k_0-1} d(Y(x_0,2^{-k}r) ,Y(x_0,2^{-(k+1)}r)) \notag \\
  &\leq& 2 \cdot 2^{-k_0} r +  \sum_{k =0}^{k_0-1}  2 \cdot 2^{-k} r \leq 2 r \sum_{k \in \N}    2^{-k}  = 4r.
\end{eqnarray}
{\sc $\diamond$ Step 4.} We prove that for every $z \in \gamma_{x, r}$
  \begin{equation}
  \label{jonesprop}
d(z,\Omega^c)\geq  \frac{29}{240} \, d(z,x). 
 \end{equation}
We start by handling case (I): we  work in the ball $B(x_0, 4r)$ and we recall the definition of $B^+(x_0, 4r)$ and of $B^-(x_0, 4r)$, given at {\sc Step 1}.  Since by assumption  $\varepsilon \leq 1/32$, we have 
$$
   16 \varepsilon r \leq \frac{r}{2} \leq d (x, \Omega^c) \leq d(x, B^-(x_0, 4r)) 
$$
and hence $x \in B^+(x_0, 4r)$. Let $\beta$ denotes the angle between  $\nu_r$ and $\nu_{4r}$,  then by Lemma \ref{angleLemma} applied with $M=4$ we get that provided $\varepsilon \leq 1/9$,  then 
$
    4 \varepsilon r \leq r \cos \beta,
$
so that $Y(x_0, r) \in B^+(x_0, 4r)$. By recalling that $x \in B^+(x_0, 4r)$, we conclude that $[x, Y(x_0, r) ] \subseteq B^+(x_0, 4r)$.

We are now ready to establish~\eqref{jonesprop}, so we fix $z \in [x, Y(x_0, r)]$.
To provide a bound from above on $d(z, x)$, we simply observe that, since both $x$ and $Y(x_0, r)$ belong to the closure of $B (x_0, 2r)$, then so does $z$ and hence 
\begin{equation}
\label{e:above}
   d(z, x) \leq 4r .
\end{equation}
Next, we provide a bound from below on $d(z, \Omega^c)$: since $z \in B^+(x_0, 4r) \subseteq \Omega$, then 
\begin{equation}
\label{e:bplus}
{d(z, \Omega^c) \ge d(z, \partial B^+(x_0, 4r))} = 
\min \Big\{  d(z, P^+  (x_0, 4r)), d (z, \partial B(x_0, 4r) )      \Big\} .
\end{equation}
First, we recall that $z \in B(x_0, 2r)$ and we provide a bound on the distance from $z$ to the  spherical part of $\partial B^+(x_0, 4r)$: 
$$
    d ( z, \partial B(x_0, 4r) ) = 4r - d(z, x_0) \ge 4r - 2r =2 r. 
$$
Next, we observe that 
$$
    d(z, P^+(x_0, 4r)) = d ( z, P(x_0, 4r) ) - 8 \varepsilon r
$$
and, since $z \in  [x, Y(x_0, r)]$, then 
$$
   d ( z, P(x_0, 4r) ) \ge \min \Big\{ d ( x, P(x_0, 4r), d ( Y(x_0, r), P(x_0, 4r) ) \Big\}. 
$$
Note that $d ( Y(x_0, r), P(x_0, 4r) )= r \cos \beta$ and, using Lemma \ref{angleLemma}, we conclude that $${d ( Y(x_0, r), P(x_0, 4r) )\geq r/2}$$ because $\varepsilon \leq 1/10$. Also, since $B^- (x_0, 4r) \subseteq \Omega^c$, then  
$$
 r/2 \leq  d(x, \Omega^c) \leq d(x, B^-(x_0, 4r))  =  d ( x, P(x_0, 4r) ) +  2 \varepsilon r 
$$
By  recalling~\eqref{e:bplus} and the inequality $\varepsilon \leq 1/600$ and by combining all the previous observations we conclude that 
\begin{equation}
\label{e:below}
\begin{split}
    d(z, \Omega^c)   & \ge d(z, \partial B^+(x_0, 4r) ) \ge     
     \min \{ d(z, P^+ (x_0, 4r) , 2r \}  =  
       \min \{ d(z, P (x_0, 4r)- 8 \varepsilon r, 2r \} \ge \\
       &  \ge   
     \min \Big\{ \min \big\{ d \big( x, P(x_0, 4r) \big), 
     d \big( Y(x_0, r), P(x_0, 4r) \big) \big\}- 8 \varepsilon r , 2r  \Big\}   \ge \\
       &  \ge    
     \min \Big\{ \min \big\{ r / 2 - 2 \varepsilon r, 
     r /2 \big\}- 8 \varepsilon r , 2r  \Big\}  = 
     \frac{r}{2} - 10 \varepsilon r \ge \frac{29}{60} r.       \\
     \end{split}
\end{equation}
Finally, by comparing~\eqref{e:below} and~\eqref{e:above} we obtain~\eqref{jonesprop}. \\
{\sc $\diamond$ Step 5.} We now establish~\eqref{jonesprop} in case (II). 

If $z \in [x, Y(x_0, 2^{-k_0}r]$, then we can repeat the argument we used in {\sc Step 4} by replacing $r$ with $2^{-k_0}r$, which satisfies 
$$
   d(x, \Omega^c) \leq   2^{-k_0} r \leq 2  d(x, \Omega^c). 
$$
Hence, we are left to consider the case when $z \in [Y(x_0,2^{-k}r), Y(x_0,2^{-(k+1)}r)]$ for some natural number $k \leq k_0 -1$. We set $\rho:= 2^{-k} r$ and we work in the ball $B(x_0, 2 \rho)$. We denote by  $\alpha$ the angle between $\nu_{2 \rho}$ and $\nu_{\rho}$, and by $\beta$ the angle between $\nu_{2 \rho}$ and $\nu_{\rho/2}$. Due to Lemma \ref{angleLemma} applied with $M=2$ and $M=4$, we know that, if $\varepsilon\leq 1/13$, then  
$$
    \rho \cos \alpha  \ge  4 \varepsilon \rho \qquad 
    \frac{1}{2} \rho \cos \beta \ge 4 \varepsilon \rho, 
$$
so that both $Y(x_0, \rho)$ and $Y(x_0, \rho/2)$ belong to $B^+(x_0, 2 \rho)$. Hence, given 
$${z \in [Y(x_0, \rho), Y(x_0, \rho/2)] } \subseteq B^+ (x_0, 2 \rho)\subseteq \Omega,$$ we have 
$
{ d(z, \Omega^c) \ge d(z, \partial B^+(x_0, 2 \rho))}.
$
The distance from $z$ to the spherical part of $\partial B^+(x_0, 2 \rho)$ is bounded from below by 
$\rho$, while the distance from $z$ to $P^+(x_0, 2 \rho)$ is bounded from below by 
$
  \frac{1}{2} \rho - 4 \varepsilon \rho \ge  \frac{1}{4} \rho
$
provided that $\varepsilon \leq 1/16.$ Hence, 
$
   d(z, \Omega^c) \ge \rho/4.
$
To provide an upper bound on $d(z, x)$ we observe that, since ${d(x, x_0)= d(x, \Omega^c) \leq 2^{-k}r}$, then both $z$ and $x$ belong to the closure of $B(x_0, \rho)$. Hence, ${d(x, z) \leq 2 \rho}$ and~\eqref{jonesprop} holds.\\
{\sc $\diamond$ Step 6.} We are finally ready to show that $\Omega$ is a Jones-flat domain. Given $x, y \in \Omega$ satisfying $d(x, y) \leq r_0/7$, there are two possible cases:
\begin{enumerate}
\item if either $d(x,\Omega^c) \ge 2 d(x,y)$ or $d(y,\Omega^c) \ge 2 d(x,y)$, then we set $\gamma : =[x,y]$. To see that $\gamma$ satisfies~\eqref{jonesprop3}, let us assume that $d(x,\Omega^c) \ge 2 d(x,y)$ (the other case is completely analogous), then $y \in B(x, d(x,y))\subseteq \Omega$ and $[x, y] \subseteq \Omega$. Also, since 
 \begin{eqnarray}
 \sup_{z \in [x,y]} \frac{d(z,x)d(z,y)}{d(x,y)} = \frac{1}{4}d(x,y), \label{geod}
 \end{eqnarray}
 then for any $z \in \gamma$, 
 $$
    d(z,\Omega^c) \geq d(x,\Omega^c) -d(z,x) \geq d(x,y)\geq 4d(z,x)d(z,y)/d(x,y). 
 $$
 Hence, $\gamma$ satisfies~\eqref{jonesprop3} provided that $\delta =4$. 
\item we are left to consider the case when both $d(x,\Omega^c) < 2 d(x,y)$ and $d(y,\Omega^c) < 2 d(x,y)$.  Denote by $x_0 \in \partial \Omega$ a point such that $d(x,\Omega^c)=d(x,x_0)$ and $y_0 \in \partial \Omega$ a point such that $d(y,y_0)=d(y,\Omega^c)$ and set 
$r:=d(x,y)\leq r_0/7$. We define
\begin{equation}
\label{e:gamma}
\gamma: = \gamma_{x,r}\cup\gamma_{y,r}  \cup [Y(x_0,r),Y(y_0,r)].
\end{equation}
{\sc Step 7} is devoted to showing that $\gamma$ satisfies~\eqref{e:lenght} and~\eqref{jonesprop3}. 
\end{enumerate}
{\sc $\diamond$ Step 7.} First, we establish~\eqref{e:lenght}: we observe that 
$$
    d \Big( Y(x_0,r),Y(y_0,r) \Big) \leq  d \Big( Y(x_0,r), x_0 \Big)  + d(x_0, x) +  d(x, y) +
    d(y, y_0) + 
    d \Big( y_0,Y(y_0,r) \Big) \leq 7r
$$
and hence by using~\eqref{jonesprop2} 
$$
     \Hh^1(\gamma) \leq  \Hh^1(\gamma_{x,r}) +   d \Big( Y(x_0,r),Y(y_0,r) \Big)  + 
      \Hh^1(\gamma_{y,r})  \leq 15r
$$
which proves \eqref{e:lenght}.

Next, we establish~\eqref{jonesprop3}: we denote by $d_{\gamma}$ the geodesic distance on the curve $\gamma$ and we observe that 
\begin{equation}
\label{e:geo} 
          \frac{d(z,y)}{15 d(x,y)} \leq \frac{d_\gamma(z,y)}{d_\gamma(x,y)}\leq 1.
\end{equation}
Hence, if $z \in \gamma_{x, r}$, then by using~\eqref{jonesprop} we obtain 
$$    d(z,\Omega^c)\geq \frac{29}{240} d(z,x)\geq \frac{29}{240 \cdot 15}
    \left(\frac{d(z,x)d(z,y)}{d(x,y)}   \right)        
$$
and we next observe $29/240 \cdot 15 \ge 5 / 60 \cdot 15 = 1/ 180.$  Since the same argument works in the case when $z\in \gamma_{y,r}$, then we are left to esablish~\eqref{jonesprop3} in the case when $z$ lies on the segment  $[Y(x_0,r),Y(y_0,r)]$. 

We first observe that 
\begin{equation}
\label{e:dxY}
   d(x_0, Y(y_0, r))  \leq
  d(x_0, x) + d(x, y) + d(y, y_0) + d(y_0,Y (y_0, r) ) \leq 6r 
\end{equation}
and hence $[Y(x_0,r),Y(y_0,r)] \subseteq B (x_0, 7r)$. Next, we note that $7r \leq r_0$ and we use~\eqref{e:below} to get 
\begin{equation}
\label{e:dYp}
  \frac{29}{60}r \leq d(Y(x_0, r), \Omega^c) \leq d(Y(x_0, r), P^-(x_0, 7r)),  
\end{equation}
hence since $\varepsilon$ is so small that $28 \varepsilon r \leq  29 r/ 60$, then we have $d(Y(x_0, r), P^-(x_0, 7r)) \ge 28 \varepsilon r$, which means that $Y(x_0, r) \in B^+(x_0, 7r)$. By repeating the same argument we get  ${Y(x_0, r) \in B^+(x_0, 7r)}$ and hence 
${[Y(x_0,r),Y(y_0,r)]  \subseteq B^+(x_0, 7r)}$. 

We fix $z \in [Y(x_0,r),Y(y_0,r)] $ and we observe that
\begin{equation}
\label{e:dzx}
\begin{split}
    d(z, x) & \leq d(z, Y(x_0, r)) + d(Y(x_0, r), x_0) + d(x_0, x)  \\
&    \leq 
    d(Y(y_0,r), Y(x_0, r)) + d(Y(x_0, r), x_0) + d(x_0, x) 
    \leq 
    7r + r + 2r =10r. \\
\end{split}
\end{equation}
Also,  
\begin{equation}
\label{e:dzo}
    d(z, \Omega^c) \ge d(z, \partial B^+(x_0, 7r)) \ge 
    \min \Big\{ d (z, \partial B(x_0, 7r) ); d(z, P^+(x_0, 7r)) \Big\}
\end{equation}
and by using~\eqref{e:dxY} we get 
$$
    d(z, \partial B(x_0, 7r)  ) \ge r.  
$$
Also, we have  
$$
    d(z, P^+(x_0, 7r))  \ge \min \Big\{   d(Y(x_0, r), P^+(x_0, 7r)), d( Y(y_0, r), P^+(x_0, 7r))   
    \Big\}  
$$
and  by recalling~\eqref{e:dYp} we get  that 
$$
    d(Y(x_0, r), P^+(x_0, 7r)) = d(Y(x_0, r), P^-(x_0, 7r)) - 28 \varepsilon r \ge 
    \frac{29}{60}r  - 28 \varepsilon r  \ge \frac{r}{3}    .
$$ 
Since $Y(y_0,r)$ satisfies the same estimate, then by recalling~\eqref{e:geo},~\eqref{e:dzx} and~\eqref{e:dzo}
we get 
$$
    d(z, \Omega^c)  \ge \frac{r}{3} \ge \frac{1}{3 \cdot 10} d(z, x) \ge 
    \frac{1}{3 \cdot 10 \cdot 15} \frac{d(z, x) d(z, y)}{d(x, y)}, 
$$
which concludes the proof because $3 \cdot 10 \cdot 15 = 450$. 
\end{proof}

\begin{remark} There are Jones-flat domains that are not Reifenberg-flat, for instance a Lipschitz domain with sufficiently big constant (for example a heavily non convex polygonal domain). Actually, Jones \cite[Theorem 3]{J} proved that, for a simply connected domain in dimension 2, being Jones-flat is equivalent to being an extension domain, which is also known to be equivalent to the fact that the boundary is a quasicircle (see the introduction of \cite{J}).
\end{remark}
%%%%%%%%%%%%%%%%%%%%%%%%%%%%%%%%%%%%%%

%%%%%%%%%%%%%%%%%%%%%%%%%%%%%%%%%%%%%%
\section{Extension properties of Reifenberg-flat domains and applications }
\label{appl} 
In this section we combine the analysis in~\cite{J} with Theorem~\ref{jonesrei} to prove that domains that are sufficiently flat in the sense of Reifenberg satisfy the extension property. We also discuss some direct consequences.  Note that in this section we always assume that $\Omega$ is connected, as Jones did in~\cite{J}. In Section~\ref{connected} we prove that the connectedness assumption can be actually removed in the case of Reifenberg flat domains. Note also that, before providing the precise extension result, we have to introduce a preliminary lemma comparing different notions of ``radius" of a given domain $\Omega$.
\subsection{``Inner radius", ``outer radius" and ``diameter" of a given domain}
We term outer radius of a nonempty set $\Omega \subseteq \R^N$ the quantity
\begin{equation}
\label{e:outrad}
        Rad (\Omega) : = \inf_{x \in \Omega}  \sup_{y \in \Omega} d(x, y),
\end{equation}
and we term inner radius the quantity 
\begin{equation}
\label{e:inrad}
        rad (\Omega) : = \sup_{x \in \Omega}  \sup \{r>0: \; B(x,r)\subset \Omega\}.
\end{equation}
The inner radius is the radius of the biggest ball that  could fit inside $\Omega$, whereas the outer radius, as seen below, is the radius of the smallest ball, centered in $\overline{\Omega}$, that contains $\Omega$.  

Also, we recall that $Diam (\Omega)$ denotes the diameter of $\Omega$, 
namely
$$
   Diam (\Omega) : = \sup_{x, y \in \Omega} d (x, y). 
$$
For the convenience of the reader, we collect some consequences of the definition in the following lemma.
\begin{lem}
\label{l:rad}
Let $\Omega$ be a nonempty subset of $\R^N$, then the following properties 
hold:
\begin{itemize}
\item[(i)] We have the formula
\begin{equation}
\label{e:rad} 
    Rad(\Omega)=\inf_{x \in \Omega} \inf\{r>0: \; \Omega \subset B(x,r)\}.
\end{equation}
 
 Also, if $Rad (\Omega) < + \infty$, then there is a point $ x \in \overline{\Omega}$ such that $\Omega \subseteq B(x, Rad (\Omega))$.
 \item[(ii)] $rad(\Omega)\leq Rad(\Omega)\leq \diam(\Omega). $
\item[(iii)] If  $\Omega$ is an $(\varepsilon, r_0)$-Reifenberg-flat domain for some $r_0 >0$ and some $\varepsilon$ satisfying $0 < \varepsilon < 1/2$, then  $r_0/4 \leq rad (\Omega)\leq Rad(\Omega) \leq \diam(\Omega) .$
\end{itemize}
\end{lem}
\begin{proof}
To establish property (i), we first observe that, if $\Omega$ is not bounded, then $Rad (\Omega) = + \infty$ and formula~\eqref{e:rad} is trivially satisfied. Also, the assumption $Rad (\Omega) < + \infty$ implies that the closure $\overline{ \Omega}$ is compact. Hence, if  $Rad (\Omega) < + \infty$, then 
\begin{equation}
\label{e:min}
    Rad (\Omega) = \min_{x \in \overline{\Omega}}  \sup_{y \in \Omega} d(x, y)
\end{equation}
and if we term $x_0\in \overline{\Omega}$ any point that realizes the minimum in \eqref{e:min} we have $\Omega \subset \overline{B}(x_0,Rad(\Omega))$. This establishes the inequality 
$$Rad(\Omega)\geq \inf_{x \in \Omega} \inf\{r>0: \; \Omega \subset B(x,r)\}.$$
To establish the reverse inequality we observe that if  $x\in \Omega$ is any arbitrary point and  $r>0$ is such that  $\Omega \subset B(x,r)$, then $\sup_{y \in \Omega}d(x,y)\leq r$.  By taking the infimum in $x$ and $r$ we conclude. This ends the proof of property (i).

To establish (ii), we focus on the case when $Rad(\Omega)<+\infty$, because otherwise $\Omega$ is unbounded and (ii) trivially holds. Hence, by relying on (i) we infer that $\Omega \subseteq B:=B(x_0, Rad (\Omega))$ for some point $x_0\in \Omega$. Given $x \in \Omega$ and $r>0$ satisfying $B(x,r)\subset \Omega$, we have $B(x,r)\subset B(x_0,Rad(\Omega))$. Hence,  $ d(x, x_0)+r\leq Rad(\Omega)$ and hence $r\leq Rad(\Omega)$.  By taking the supremum in $r$ and $x$ we get finally $rad(\Omega)\leq Rad(\Omega)$. The inequality $Rad(\Omega)\leq \diam(\Omega)$ directly follows from the two definitions.

Given (ii), establishing property (iii) amounts to show that 
 \begin{eqnarray}
 rad(\Omega)\geq r_0/4. \label{amontrer17}
 \end{eqnarray}
We can assume with no loss of generality that $\partial \Omega\not = \emptyset$, otherwise $\Omega=\R^N$ and~\eqref{amontrer17}  trivially holds in this case (we recall that the case $\Omega=\emptyset$ is ruled out by the definition of Reifenberg-flat domain). 

Hence, we fix $y \in \partial \Omega$, denote by $P(y, r_0)$ the hyperplane in the definition and let $\vec \nu$ be its normal vector. We choose the orientation of $\vec \nu$ in such a way that 
\begin{equation}
\label{e:incl}
    \{ z + t \nu: \; z \in P(y, r_0), \; t \ge 2 \varepsilon r \} \cap B(y, r_0 ) \subseteq \Omega. 
\end{equation}
Since $d_H (P(y, r_0) \cap B(y, r_0), \partial \Omega \cap B(y, r_0)    ) \leq \varepsilon r$, then from~\eqref{e:incl} we infer that actually 
$$
      \{ z + t \nu: \; z \in P(y, r_0), \; t \ge  \varepsilon r \} \cap B(y, r_0 ) \subseteq \Omega. 
$$
By recalling $\varepsilon < 1/2$, we infer that there is $x\in \Omega$ such that $B(x,r_0/4)\subset \Omega$ and this establishes \eqref{amontrer17}.
 \end{proof}

\subsection{Extension properties and applications}
The following extension property of Reifenberg flat domains is established by combining Theorem~\ref{jonesrei} above with Jones'analysis (Theorem 1 in~\cite{J}).
\begin{cor} \label{ext} 
Let $\Omega \subseteq \mathbb R^N$ be a connected, $(\varepsilon,r_0)$-Reifenberg-flat domain. If $\varepsilon \leq 1/600$, then, for every $p \in [1,+\infty] $, there is an extension operator $E : W^{1,p}(\Omega) \to W^{1,p}(\R^N)$ satisfying
$$\|E(u)\|_{W^{1,p}(\R^N)} \leq C \|u\|_{W^{1,p}(\Omega)},$$
where the constant $C$ only depends on $N$, $p$,  and $r_0$.
\end{cor}
\begin{proof} 
The corollary is a direct application of~\cite[Theorem 1]{J}. 

The only nontrivial point we have to address is that, in general, the norm of the extension operator $E$ depends on $Rad (\Omega)$, see for examples the statements of Jones' Theorem provided in the paper by Chua~\cite{Chua-indiana} and in the very recent preprint by Brewster, D. Mitrea, I. Mitrea and M. Mitrea~\cite{BrewsterMitreas}. Note that in Jones' original statement the dependence on the radius was not mentioned because the radius was fixed (see the remark at the top of page 76 in~\cite{J}). 

However, by applying for example the remarks in~\cite[pages 9 and 10]{BrewsterMitreas} to Reifenberg-flat domains, we get that the norm of $E$ is bounded by $C(N, p, r_0, M)$ if $ 1/ Rad (\Omega) \leq M$. By recalling that $r_0/4 \leq Rad(\Omega)$, we finally infer that the bound on the  norm of the extension operator only depends on $N, p$ and $r_0$ and this concludes the proof.
\end{proof}

\begin{remark} 
\label{chuachrist}To simplify the exposition, we chose to only state the extension property for classical Sobolev Spaces. However, the extension property also applies to other classes of spaces. For instance, Chua~\cite{Chua-indiana,Chua-illinois,Chua-Canad}, extended  Jones' Theorem to weighted Sobolev spaces. These spaces are defined by replacing the Lebesgue measure by a weighted measure $\omega dx$, where $\omega$ is a function satisfying suitable growth conditions and Poincar\'e inequalities.  Also, Christ~\cite{Christ} established the extension property for Sobolev spaces of fractional order. 

The results of both Christ~\cite{Christ} and Chua~\cite{Chua-indiana,Chua-illinois,Chua-Canad} apply to Jones-flat domains, hence by relying on Theorem~\ref{jonesrei} we infer that they apply to $(1/600,r_0)$-Reifenberg-flat domains as well. 
\end{remark}
As a consequence of Corollary \ref{ext} we get that the classical Rellich-Kondrachov Theorem holds in Reifenberg-flat domains. 
\begin{prop}
\label{embedding}
Let $\Omega \subseteq \mathbb R^N$ be a bounded, connected $(\varepsilon,r_0)$-Reifenberg-flat domain and assume $0<\varepsilon \leq 1/600$. 

If $1\leq p < N$, set $p^*:=\displaystyle{\frac{Np}{N-p}}$.  Then the Sobolev space $W^{1,p}(\Omega)$ is continuously embedded in the  space $L^{p^*}(\Omega)$ and is compactly embedded in $L^q(\Omega)$  for every $1\leq q < p^*$.   

If $p \ge N$, then the Sobolev space $W^{1,N}(\Omega)$ is continuously embedded in the  space $L^{\infty}(\Omega)$ and is compactly embedded
$L^q(\Omega)$  for every $q \in [1, + \infty[$. 

Also, the norm of the above embedding operators  only depends on $N$, $r_0$, $q$, $p$ and $Rad(\Omega)$. 
\end{prop}    
\begin{proof} We first use the extension operator provided by Corollary \ref{ext} and then we apply the classical Embedding Theorem in a ball of radius $Rad(\Omega)$ containing $\Omega$ (see property (i) in the statement of Lemma~\ref{l:rad}).
\end{proof}

As an example of application of Proposition \ref{embedding}, we establish  a   uniform bound on the $L^{\infty}$ norm of Neumann eigenfunctions defined in Reifenberg-flat domains. We use this bound in the companion paper~\cite{lms2}. Here is the precise statement. We recall that we term ``Neumann eigenfunction" an eigenfunction for the Laplace operator subject to homogeneous Neumann conditions on the boundary of the domain. 
\begin{prop}\label{uestim} Let $\Omega \subseteq \mathbb R^N$ be a bounded, connected, $(\varepsilon,r_0)$-Reifenberg-flat domain and let $u$ be a Neumann eigenfunction associated to the eigenvalue $\mu$. If $\varepsilon \leq 1/600$, then $u$ is bounded and 
\begin{eqnarray}
\label{e:linfty}
\|u \|_{L^\infty(\Omega)}\leq C (1+\sqrt{\mu})^{\gamma(N)} \|u\|_{L^{2}(\Omega)}, \label{desirein}
\end{eqnarray}
where $\gamma(N)=\max\big\{ \frac{N}{2},\frac{2}{N-1} \big\}$ and $C=C(N, r_0,Rad(\Omega))$.
\end{prop}
\begin{proof} By using classical techniques coming from the regularity theory for elliptic operators, Ross~\cite[Proposition 3.1]{marty} established~\eqref{desirein} in the case of Lipschitz domains. However, in~\cite{marty} the only reason why one needs the regularity assumption on the domain $\Omega$ is to 
use the Sobolev inequality
\begin{equation}
\label{e:sobsob}
     \|u\|_{L^{2^*}(\Omega)}\leq C(\|u\|_{L^2(\Omega)}+\|\nabla u\|_{L^2(\Omega)}), \qquad C = C(N,r_0, Rad(\Omega))
\end{equation}
as the starting point for a bootstrap argument. Since Proposition \ref{embedding} states that~\eqref{e:sobsob} holds if $\Omega$ is a bounded Reifenberg-flat domain, then the proof in~\cite{marty} can be extended to the case of Reifenberg-flat domains.
\end{proof}

\begin{remark}
An inequality similar to \eqref{desirein} holds for Dirichlet eigenfunctions.  We emphasize that the boundedness of Dirichlet eigenfunctions, unlike the boundedness of Neumann eigenfunctions, does not require any regularity assumption on the domain $\Omega$, see for instance \cite[Lemma 3.1.]{davies2} for a precise statement.
 \end{remark}

%%%%%%%%%%%%%%%%%%%%%%%%%%%%%%%%%%%%%%%%%

\section{Connected components of Reifenberg-flat domains}
\label{connected}
In the previous section we have always assumed that the domain $\Omega$ is connected. 
We now show that the results we have established can be extended to general (i.e., not necessarily connected) Reifenberg-flat domains. Although extension of the result of Jones ~\cite{J} to non-connected domains were already widely known in the literature, we decided to provide here a self-contained proof. In this way, we obtain results on the structure of Reifenberg-flat domains that may be of independent interest.   

We first show that any sufficiently flat Reifenberg-flat domain is finitely connected and we establish a quantitative bound on the Hausdorff distance between two connected components.  
\begin{prop} \label{topol1}  Let $\Omega \subseteq \R^N$ be a bounded, $(\varepsilon,r_0)$-Reifenberg flat domain and we assume $\varepsilon \leq 20^{-N}$. Then $\Omega$ has a finite number of nonempty, open and disjoint connected components $U_1$, ... , $U_{n}$, where 
\begin{equation}
\label{e:n}
n \leq \frac{20^N}{\omega_N} \frac{|\Omega|}{r_0^N}. 
\end{equation}
Moreover, if $i \neq j$, then for every $z \in \partial U_i$ we have 
\begin{equation}
\label{e:sep}
        d(z, U_j) > r_0/70. 
\end{equation}
\end{prop}

\begin{proof} We proceed according to the following steps.\\
{\sc $\diamond$ Step 1} We recall that any nonempty open set $\Omega\subseteq \R^N$ can be decomposed as
\begin{equation}
\label{e:dec}
         \Omega: = \bigcup_{i \in I} U_i, 
\end{equation}
where the connected components $U_i$ satisfy
\begin{itemize}
\item for every $i \in I$, $U_i$ is a nonempty, open, arcwise connected set which is also closed in $\Omega$. Hence, in particular, $\partial U_i \subseteq \partial \Omega$. 
\item $U_i \cap U_j = \emptyset$ if $ i \neq j$. 
\end{itemize}
Indeed, for any $x \in \Omega$ we can define 
$$
    U_x : = \big\{ y \in \Omega: \; \text{there is a continuous curve 
    $\gamma: [0, 1] \to \Omega$ such that $\gamma(0)=x$ and $\gamma(1)=y$}\big\}
$$
and observe that any $U_x$ is a nonempty, open, arcwise connected set which is also closed in $\Omega$. Also, given two points $x, y \in \R^N$, we have either $U_x =U_y$ or $U_x \cap U_y = \emptyset$. \\
{\sc $\diamond$ Step 2} Let $\Omega$ as in the statement of the proposition, and let the family $\{U_i \}_{i \in I}$ be as in~\eqref{e:dec}.  We fix $i \in I$ and we prove that $|U_i| \ge C(r_0, N) $. This straightforwardly implies that $\sharp I \leq C(|\Omega|, r_0, N)$. 

Since $U_i$ is bounded, then $\partial U_i \neq \emptyset$: hence, we can fix a point $\tilde x \in \partial U_i$, and a sequence $\{ x_n \}_{n \in \mathbb N} $ such that  $x_n \in U_i$ and $x_n \to \tilde x$ as $n \to + \infty$. We recall that $\partial U_i \subseteq \partial \Omega$ and we infer that, for any $n \in \mathbb N$, the following chain of inequalities holds:
$$
    d(x_n, \partial U_i) = d(x_n, U_i^c) \leq d (x_n, \Omega^c) = d(x_n, \partial \Omega) \leq d(x_n, \partial U_i),  
$$
which implies $d (x_n, \Omega^c) = d(x_n, \partial U_i)$. We fix $n$ sufficiently large such that 
$d(x_n, \tilde x) \leq r_0 /7$, so that 
$$d(x_n, \Omega^c)=  d(x_{n},\partial U_i)\leq r_0/7.$$ We term $\Gamma:=\gamma_{x_{n},r_0/7}$ the polygonal curve constructed as in Step 2 of the proof of  Theorem \ref{jonesrei} and we observe that, if $\varepsilon \leq 1/32$, then~\eqref{jonesprop} holds and $\Gamma \subseteq \Omega$ and hence, by definition of $U_i$, $\Gamma \subseteq U_i$. We use the same notation as in {\sc Step 1} of the proof of Theorem~\ref{jonesrei} and we recall that $\Gamma$ connects $x_n$ to some point $Y(x_0, r_0/7)$, defined with some $x_0\in \partial \Omega$. Hence, in particular, $Y(x_0, r_0/7) \in U_i$ and this implies that $B^+(x_0, r_0/7) \subseteq U_i$ because $B^+(x_0, r_0/7) $ is connected. This finally yields
$$
    |U_i| \ge |B^+(x_0, r_0/7)| \ge \omega_N \Big( \frac{r_0}{14} (1 - 2 \varepsilon) \Big)^N\geq \omega_N \Big( \frac{9r_0}{140} \Big)^N\geq \omega_N \Big( \frac{r_0}{20} \Big)^N,
$$ 
 because $\varepsilon\leq 1/20$. We deduce that 
$$
    \sharp I \leq \frac{20^N}{\omega_N} \frac{|\Omega|}{r_0^N}.
$$
{\sc $\diamond$ Step 3} We establish the separation property~\eqref{e:sep}. 

We set $r_1:= r_0 /70$ and we argue by contradiction, assuming that there are $z \in \partial U_i$, $y \in \partial U_j$ such that
$$
    d(z, U_j) = d (z, \partial U_j) = d(z, y) \leq r_1. 
$$
Let $\{ z_n \}_{n \in \N}$ and $\{y_n \}_{n \in \N}$ be sequences in $U_i$ and $U_j$ converging to $z$ and $y$, respectively. We fix $n$ sufficiently large such that 
$$
    d(z_n, \partial U_i) \leq d(z_n, z) \leq r_1 \leq r_0/14 
$$
and we term $\bar z$ be a point in $\partial U_i$ satisfying $d(z_n, \bar z) = d(z_n, \partial U_i)$ (if there is more than one such $\bar z$, we arbitrarily fix one). By arguing as in {\sc Step 2}, we infer that $B^+ (\bar z, r_0/14) \subseteq U_i$. 
Next, we do the same for $U_j$, namely we fix $m$ sufficiently large that 
$$
    d(y_m, \partial U_j) \leq d(y_m, y) \leq r_1 \leq r_0/7, 
$$
we let $\bar y$ be a point in $\partial U_j$ satisfying $d(y_m, \bar y) = d(y_m, \partial U_j)$ and, by arguing as in {\sc Step 2}, we get that $B^+ (\bar y, r_0/7) \subseteq U_j$. Also, we note that 
$$
    d(\bar z, \bar y) \leq d(\bar z, z_n) + d(z_n, z) + d(z, y) + d(y, y_m) + d(y_m, \bar y) 
    \leq 5 r_1. 
$$
Since $r_1 = r_0 /70$, then 
$B^+ (\bar z, r_0/14) \subseteq B (\bar z, r_0/14) \subseteq  B (\bar y, r_0/7)$. We observe that 
\begin{eqnarray}
{B^+ (\bar z, r_0/14) \cap B^- (\bar y, r_0/7)= \emptyset} \label{totor1}
\end{eqnarray}
since by construction 
$B^+ (\bar z, r_0/14) \subseteq \Omega$ and $B^- (\bar y, r_0/7) \subseteq \Omega^c$. Also, by recalling that
$$
   B^+ (\bar z, r_0/14) \subseteq U_i, \qquad  B^+ (\bar y, r_0/7) \subseteq U_j \; \; \; \;\text{and} \;  \; \; \;
U_i \cap U_j = \emptyset,
$$
we have that 
\begin{eqnarray}
B^+ (\bar z, r_0/14)  \cap B^+ (\bar y, r_0/7) = \emptyset \label{totor2}
\end{eqnarray}
By combining~\eqref{totor1} and~\eqref{totor2} we get  
\begin{eqnarray}
 B^+ (\bar z, r_0/14) \subseteq B (\bar y, r_0/7) \setminus  \big( B^+ (\bar y, r_0/7)  \cup B^- (\bar y, r_0/7) \big). \label{totor3}
 \end{eqnarray}
We now use the inequality 
\begin{eqnarray}
\omega_{N}\geq \omega_{N-1} \frac{1}{2^{N-1}}, \label{amontrerlater}
\end{eqnarray}
which will be proven later. By relying on \eqref{amontrerlater} and by recalling that $\varepsilon\leq 20^{-N}\leq 1/20$ we obtain
\begin{eqnarray}
    | B^+ (\bar z, r_0/14) | \ge \omega_N \left( \frac{r_0}{28} (1-2 \varepsilon) \right)^N \geq 2\omega_{N-1}\left( \frac{9r_0}{560}  \right)^N \notag
 \end{eqnarray}
 and 
\begin{eqnarray}
    \Bigg| 
           B (\bar y, r_0/7) \setminus  \big( B^+ (\bar y, r_0/7)  \cup B^- (\bar y, r_0/7) \big)
    \Bigg| \leq 4 \varepsilon \omega_{N-1} \left( \frac{r_0}{7} \right)^N\leq 2\omega_{N-1}\left( \frac{2r_0}{140} \right)^N, \notag
\end{eqnarray}
which contradicts \eqref{totor3} since $2/140< 9/560$.

To finish the proof we are thus left to establish~\eqref{amontrerlater}. To do this,  we use the relation
\begin{eqnarray}
\omega_N&=&\omega_{N-1}\int_{-1}^1\big(\sqrt{1-x^2}\big)^{N-1}dx \notag.
\end{eqnarray}
This implies that, for  any $\lambda�\in (0,1)$, we have 
\begin{eqnarray}
\omega_N&\geq & \omega_{N-1}2\int_{0}^\lambda \big(\sqrt{1-x^2}\big)^{N-1}dx \notag \\
&\geq & \omega_{N-1}2 \lambda \big(\sqrt{1-\lambda^2}\big)^{N-1} \notag
\end{eqnarray}
By choosing $\lambda={\sqrt{3}}/{2}$ we obtain the inequality
$$\omega_{N}\geq \omega_{N-1}\frac{\sqrt{3}}{2^{N-1}}\geq \omega_{N-1}\frac{1}{2^{N-1}},$$
and this concludes the proof. 
\end{proof}
By relying on Proposition~\ref{topol1} we can now remove the connectedness assumption in the statement of Proposition~\ref{ext}. 
\begin{cor}  Let $N\geq 2$ and $\Omega \subseteq \R^N$ be a bounded, $(\varepsilon,r_0)$-Reifenberg flat domain with $\varepsilon \leq \min(20^{-N},1/600)$. Then for every $p \in [1, + \infty]$ there is an extension operator 
\begin{equation}
          E: W^{1, p} (\Omega) \to W^{1, p} (\R^N)         
\end{equation}
whose norm is bounded by a constant which only depends on $N$, $p$, and $r_0$.
\end{cor}

\begin{proof}  We employ the same notation as in the statement of Proposition~\ref{topol1} and we fix a connected component $U_i$. By recalling that $\partial U_i \subseteq \partial \Omega$ and the separation property~\eqref{e:sep}, we infer that $U_i$ is itself a $(\varepsilon, r_0/140)$-Reifenberg flat domain.  Since by definition $U_i$ is connected, we can apply Proposition~\ref{ext} which says  that, for every $p \in [1, + \infty]$, there is an extension operator
$$
    E_i: W^{1, p} (U_i) \to W^{1, p} (\R^N)
$$
whose norm is bounded by a constant which only depends on $N$, $p$ and $r_0$.

%%%%%%%%%%%%%%%%%%%%%

In order to ``glue together" the extension operators $E_1, \dots, E_n$ we proceed as follows. Given $i=1, \dots, n$, we set $\delta:=r_0/280$ and we introduce the notation
$$
    U_i^{\delta}:= \big\{ x \in \R^N: \; d(x, U_i ) < \delta  \big\}.
$$
Note that the separation property~\eqref{e:sep} implies that $U_i^{2 \delta} \cap U_j^{2 \delta} = \emptyset$ if $i \neq j$. 

We now construct suitable cut-off functions $\varphi_i$, $i=1, \dots, n$. Let $\ell: [0, + \infty[ \to [0, 1]$ be the auxiliary function defined by setting 
\begin{equation*}    
\ell (t) : = 
\left\{  
\begin{array}{ll} 
         1 & \text{if $t \leq \delta$} \\
         \displaystyle{1 + \frac{\delta - t }{\delta}  }
         & \text{if $\delta \leq t \leq 2 \delta $}   \\
         0  & \text{if $t \ge 2 \delta$} \\
\end{array} 
\right. 
\end{equation*}
We set $\varphi_i(x): = \ell \big(d(x, U_i) \big)$ and we recall that the function $x \mapsto d(x, U_i)$ is 1-Lipschitz and that $\delta = r_0 / 280$. Hence, the function $\varphi_i$ satisfies the following properties:
\begin{equation}
\label{e:varphi}
        0 \leq \varphi_i(x) \leq 1, \; \; 
         |\nabla \varphi_i(x)| \leq  C(r_0) \; \;
        \forall x \in \R^N,
         \quad \varphi_i \equiv 1 \; \textrm{on $U_i$}, \quad 
        \varphi_i \equiv 0 \; \textrm{on $\R^N \setminus U^{2 \delta}_i$}.  
     \end{equation}
We then define 
$
    E: W^{1, p} (\Omega) \to W^{1, p} (\R^N)
$
by setting 
$$
    E(u) : = \sum_{i=1}^n E_i(u)(x) \varphi_i(x).
$$
We recall that the sets $U_1, \dots, U_n$ are all pairwise disjoint, we focus on the case $p < + \infty$ and we get 
\begin{equation*}
\begin{split}
         \| E(u)\|_{L^p(\R^N)}  & =
         \left( \int_{\R^N}   \left|  \sum_{i=1}^n E_i(u)(x) \varphi_i(x) dx \right|^p \right)^{1/p}  \leq 
            \sum_{i=1}^n \left( \int_{U_i^{2 \delta}}  | E_i(u)(x) \varphi_i(x) |^p dx \right)^{1/p}
             \\ 
             & \leq 
              \sum_{i=1}^n \| E_i(u) \|_{L^p(\R^N)} 
             \leq   \sum_{i=1}^nC(N, p, r_0) \|u\|_{W^{1,p}(U_i)} \\
             &\leq 
              C(N, p, r_0) \| u \|_{W^{1, p} (\Omega)}. \phantom{\int_{\Omega}}\\
\end{split}
\end{equation*}
Also, by using the bound on $|\nabla \varphi_i|$ provided by~\eqref{e:varphi}, we get 
\begin{equation*}
\begin{split}
         \| \nabla E(u)\|_{L^p(\R^N)}  & =
         \left( \int_{\R^N}   \left|  \sum_{i=1}^n 
         \big( \nabla E_i(u)(x) \varphi_i(x)  + E_i(u)(x) \nabla \varphi_i(x) \big) dx 
         \right|^p \right)^{1/p} \\ &
          \leq 
            \sum_{i=1}^n \left( \int_{U_i^{2 \delta}}  | \nabla E_i(u)(x) \varphi_i(x) |^p dx \right)^{1/p} 
            +  \sum_{i=1}^n \left( 
            \int_{U_i^{2 \delta}}  | E_i(u)(x) \nabla \varphi_i(x) |^p dx \right)^{1/p}  
            \\ & \leq   \sum_{i=1}^n \|\nabla E_i(u)\|_{L^p(\R^N)} 
            +  C(r_0) \sum_{i=1}^n \| E_i(u) \|_{L^p(\R^N)}  \\        
             &\leq 
              C(N, p, r_0) \| u \|_{W^{1, p} (\Omega)}. \phantom{\int_{\Omega}}\\               
\end{split}
\end{equation*}
The proof in the case $p = \infty$ is a direct consequence of the bounds on the norm of $ E_i $ and on the uniform norms of $\varphi_i$ and $\nabla \varphi_i$. This concludes the proof of the corollary.  
\end{proof}

%%%%%%%%%%%%%%%%%%%%%%%%%%%%%%%%%%%%%%%%%

\section{On the Hausdorff distance between Reifenberg-flat domains}
\label{s:hausdorff}
We end this paper  by comparing different ways of measuring the ``distance" between Reifenberg-flat domains. 
\subsection{Comparison between different Hausdorff distances.}
This subsections aims at comparing the Hausdorff distances $d_H(X, Y)$, $d_H(X^c, Y^c)$ and $d_H(\partial X, \partial Y)$, where $X$ and $Y$ are subsets of $\R^N$. 

First, we exhibit two examples showing that, in general, neither $d_H (X, Y)$ controls $d_H(X^c, Y^c)$ nor $d_H (X^c, Y^c)$ controls $d_H(X, Y)$.
We term $B:=B(1, \vec 0)$ the unit ball and we consider the two perturbations $A$ and $C$ as represented in Figure~\ref{fig2}. 
\begin{figure}[h]
\begin{center}
\input{drawing.tex}
\end{center}
\caption{}
\label{fig2}
\end{figure}

Next, we exhibit an example showing that, in general, $d_H(\partial X, \partial Y)$ controls neither $d_H(X, Y)$ nor $d_H(X^c, Y^c)$. Let $X:=B(R, \vec 0)$ and $Y:=
B (R + \varepsilon, \vec 0)  \setminus B(R, \vec 0)$, then 
$$
   \varepsilon = d_H (\partial X, \partial Y) < < d_H(X, Y)= d_H(X^c, Y^c)= R.
 $$   
Also, note that the examples represented in Figure~\ref{fig2} show that, in general, neither $d_H(X, Y)$ nor $d_H(X^c, Y^c)$ controls $d_H(\partial X, \partial Y).$ Indeed, $d_H(\partial A, \partial B) \simeq 1$ and $d_H (\partial C, \partial B) \simeq 1$. 

However, if $X$ and $Y$ are two sufficiently close Reifenberg-flat domains, then we have the following result. 
\begin{lem}\label{airport2} Let $X$ and $Y$ be two $(\varepsilon, r_0)$-Reifenberg-flat domains satisfying  ${d_H(\partial X, \partial Y)\leq 2 r_0}$.  Then
\begin{equation}
\label{e:june}
d_H(\partial X, \partial Y) 
\leq  \frac{4}{1-2 \varepsilon} \min \big\{d_{H}(X,Y),d_{H}(X^c,Y^c)\big\}. 
\end{equation}
\end{lem}

\begin{proof}
Just to fix the ideas, assume that $d_H (\partial X, \partial Y) = \sup_{x \in \partial X} d(x, \partial Y)$. Since by assumption $d_H (\partial X, \partial Y) < + \infty$, then for every $h>0$ there is $x_h \in \partial X$ such that 
$$
    d_H (\partial X, \partial Y) - h \leq d_h : = d (x_h, \partial Y) \leq 
    d_H (\partial X, \partial Y). 
$$ 
Note that $\partial Y \cap B(x_h, d_h /2) = \emptyset$ and hence either (i) ${B(x_h, d_h /2) \subseteq Y}$ or
(ii) ${B(x_h, d_h /2) \subseteq Y^c}$. 

First, consider case (i): let $P(x_h, d_h/2)$ be the hyperplane prescribed by the definition of Reifenberg flatness, then by Lemma~\ref{l:ii} we can choose the orientation of the normal vector $\nu$ in such a way that 
$$
    B^- (x_h, d_h/2) : = \left\{ z +t  \nu: \;   z \in P(x_h, d_h/2), \, t \ge  \varepsilon d_h   \right\} \cap B(x_h, d_h/2)  \subseteq X^c
$$
and 
$$
       B^+ (x_h, d_h/2)  : =
      \left\{ z -t  \nu: \;   z \in P(x_h, d_h/2), \, t \ge  \varepsilon d_h   \right\} \cap B(x_h, d_h/2)  \subseteq  X. 
$$
Fix the point 
$$
\bar z:= x_h +  \frac{\big( 1+ 2\varepsilon \big) d_h}{4} \nu,
$$ 
then we have 
$$
   B \left(  \bar z,  \frac{\big( 1 - 2 \varepsilon \big) d_h}{4}  \right) \subseteq  B^- (x_h, d_h/2)  \subseteq X^c \cap Y
$$
and hence 
$$
    d_H (X^c, Y^c) \ge  \sup_{z \in X^c }  d( z, Y^c) \ge d( \bar z, Y^c) \ge \frac{(1-2 \varepsilon)d_h}{4}
$$
and 
$$
    d_H (X, Y) \ge \sup_{z \in Y} d(z, X) \ge d(\bar z, X) \ge \frac{(1- 2\varepsilon)d_h}{4}. 
$$
Since case (ii) can be tackled in an entirely similar way, by the arbitrariness of 
$h$ we deduce that 
\begin{equation}
\label{e:may}
    d_H(\partial X, \partial Y) 
\leq  \frac{4}{1-2 \varepsilon} d_{H}(X,Y).
\end{equation}
The proof of~\eqref{e:june} is concluded by making the following observations:
\begin{itemize}
\item if $X$ is an $(\varepsilon, r_0)$-Reifenberg flat domain, then $X^c$ is also an $(\varepsilon, r_0)$-Reifenberg flat domain. 
\item $\partial X = \partial X^c$ and $\partial Y = \partial Y^c$.
\end{itemize}
Hence, by replacing in~\eqref{e:may} $X$ with $X^c$ and $Y$ with $Y^c$ we obtain~\eqref{e:june}. \end{proof}
\subsection{Comparison between the Hausdorff distance and the measure of the symmetric difference} This subsection aims at comparing the Hausdorff distances $d_H(X, Y)$ and $d_H (X^c, Y^c)$ with the Lebesgue measure of the symmetric difference, $|X \triangle Y|$. As usual, $X$ and $Y$ are subsets of $\R^N$. The results we state are applied in~\cite{lms2} to the stability analysis of the spectrum of the Laplace operator with Neumann boundary conditions.

First, we observe that the examples illustrated in Figure~\ref{fig2} show that, in general, $|X \triangle Y|$ controls neither $d_H(X, Y)$ nor $d_H(X^c, Y^c)$. Indeed, $|A \triangle B| \simeq \varepsilon$ and $|C \triangle B| \simeq \varepsilon.$ However, if 
$X$ and $Y$ are two sufficiently close Reifenberg-flat domains, then the following result hold. 
\begin{lem}\label{airport1} Let $X$ and $Y$ be two $(\varepsilon, r_0)$-Reifenberg-flat domains in $\R^N$. 

Then the following implications hold:
\begin{enumerate}
\item if ${d_H(X, Y)\leq 4 r_0}$, then 
\begin{equation}
d_{H}(X,Y)  \leq   \frac{8}{(1-2 \varepsilon) } \left( \frac{ |X \triangle Y|}{\omega_N} \right)^{1/N}. \label{mardi}
\end{equation}
\item If ${d_H( X^c,  Y^c )\leq 4 r_0}$, then 
\begin{equation}
\label{mardi2}
d_{H}(X^c,Y^c)  \leq    \frac{8}{(1-2 \varepsilon) } \left( \frac{ |X \triangle Y|}{\omega_N} \right)^{1/N}. 
\end{equation}
\end{enumerate}
In both the previous expressions, $\omega_N$ denotes the measure of the unit ball in $\R^N$. 
\end{lem}
\begin{proof} The argument relies on ideas similar to those used in the proof of Lemma~\ref{airport2}. 

We first establish~\eqref{mardi}. Just to fix the ideas, assume that 
$
   d_H (X, Y) = \sup_{x \in X} d(x, Y)
$ 
and note that by assumption $d_H(X, Y) < + \infty$. Hence, for every $h>0$ there is $x_h \in X$ such that 
$$
    d_H(X, Y)  - h \leq d_h : = d(x_h, Y) \leq d_H(X, Y) 
$$
Note that, by the very definition of 
$d(x_h, Y)$, we have  
$
    B \left( x_h, d_h \right) \subseteq Y^c. 
$
We now separately consider two cases: if $B(x_h, d_h /2) \subseteq X$, then 
$$
    B (x_h, d_h /2 ) \subseteq X \cap Y^c \subseteq |X \triangle Y|
$$
and hence 
$$
    \omega_N \left( \frac{d_h}{2} \right)^N \leq |X \triangle Y|,
$$
and by the arbitrariness of $h$ this implies~\eqref{mardi}. 

Hence, we are left to consider the case when there is $x_0 \in B (x_h, d_h /2 ) \cap \partial X.$ We make the following observations: first, 
\begin{equation}
\label{e:iy}
    B(x_0, d_h/4) \subseteq B(x_h, d_h) \subseteq Y^c. 
\end{equation}
Second, since $d_h/4 \leq d_H(X, Y)/4 \leq r_0$, then we can apply the definition of Reifenberg-flatness in the ball $B(x_0, d_h/4)$. Let $P(x_0, d_h/4)$ be the hyperplane provided by property (i) in the definition, and let $\nu_0$ denote the normal vector.  By relying on Lemma~\ref{l:ii}
we infer that we can choose the orientation of $\nu_0$ in such a way that 
\begin{equation}
\label{e:ix}
    B \left(x_0 + \frac{(1 + 2 \varepsilon)d_h}{ 8}  \, \nu_0,    \frac{(1 - 2 \varepsilon)d_h}{ 8 } \right) \subseteq X \cap B(x_0, d_h/4) 
\end{equation}
By combining~\eqref{e:iy} and~\eqref{e:ix} we infer that 
$$
    \omega_N \left( \frac{(1 - 2 \varepsilon)d_h}{ 8 }   \right)^N \leq |X \cap Y^c| \leq |X \triangle Y|
$$
and by the arbitrariness of $h$ this completes the proof of~\eqref{mardi}. 

Estimate~\eqref{mardi2} follows from~\eqref{mardi} by relying on  the following two observations:
\begin{itemize}
\item $X \triangle Y = ( X^c \cap Y ) \cup ( X \cap Y^c) = X^c \triangle Y^c$. 
\item if $X$ is an $(\varepsilon, r_0)$-Reifenberg flat domain, then $X^c$ is also an $(\varepsilon, r_0)$-Reifenberg flat domain. 
\end{itemize}
Hence, by replacing in ~\eqref{mardi} $X$ with $X^c$ and $Y$ with $Y^c$ 
we get~\eqref{mardi2}. 
\end{proof}

\section{Acknowledgements} 

The authors wish to thank Tatiana Toro and Guy David for several conversations on Reifenberg flat domains. E. Milakis was supported by the Marie Curie International Reintegration Grant No 256481 within the 7th European Community Framework Programme. Part of this work was done when L.V. Spinolo was affiliated to the University of Zurich, Switzerland, which she thanks for the kind hospitality. 
\bibliographystyle{plain}
\bibliography{biblio_geom}

\def\cprime{$'$} \def\cprime{$'$}
\begin{thebibliography}{10}

\bibitem{BrewsterMitreas}
K.~Brewster, D.~Mitrea, I.~Mitrea, and M.~Mitrea.
\newblock {Extending Sobolev Functions with Partially Vanishing Traces from
  Locally (epsilon,delta)-Domains and Applications to Mixed Boundary Problems}.
\newblock {\em Preprint 2012}, also arXiv:1208.4177.

\bibitem{w1}
S.~Byun and L.~Wang.
\newblock Elliptic equations with {BMO} nonlinearity in {R}eifenberg domains.
\newblock {\em Adv. Math.}, 219(6):1937--1971, 2008.

\bibitem{w2}
S.~Byun and L.~Wang.
\newblock Gradient estimates for elliptic systems in non-smooth domains.
\newblock {\em Math. Ann.}, 341(3):629--650, 2008.

\bibitem{w3}
S.~Byun, L.~Wang, and S.~Zhou.
\newblock Nonlinear elliptic equations with {BMO} coefficients in {R}eifenberg
  domains.
\newblock {\em J. Funct. Anal.}, 250(1):167--196, 2007.

\bibitem{Calderon}
A.~P. Calderon.
\newblock Lebesgue spaces of differentiable functions and distributions.
\newblock {\em Proc. Sympos. Pure Math.}, IV:33--49, 1961.

\bibitem{Christ}
M.~Christ.
\newblock The extension problem for certain function spaces involving
  fractional orders of differentiability.
\newblock {\em Arkiv Mat.}, 22(1-2):63--81, 1984.

\bibitem{Chua-indiana}
S.-K. Chua.
\newblock Extension theorems on weighted {S}obolev spaces.
\newblock {\em Indiana Univ. Math. J.}, 41(4):1027--1076, 1992.

\bibitem{Chua-illinois}
S.-K. Chua.
\newblock Some remarks on extension theorems for weighted {S}obolev spaces.
\newblock {\em Illinois J. Math.}, 38(1):95--126, 1994.

\bibitem{Chua-Canad}
S.-K. Chua.
\newblock {Extension Theorems on Weighted Sobolev Spaces and Some
  Applications}.
\newblock {\em Canad. J. Math.}, 58(3):492--528, 2006.

\bibitem{d0}
G.~David.
\newblock Approximation of a {R}eifenberg-flat set by a smooth surface.
\newblock {\em In preparation}.

\bibitem{d1}
G.~David.
\newblock H\"older regularity of two-dimensional almost-minimal sets in {$\Bbb
  R^n$}.
\newblock {\em Ann. Fac. Sci. Toulouse Math. (6)}, 18(1):65--246, 2009.

\bibitem{d2}
G.~David.
\newblock ${C}^{1+\alpha}$-regularity for two dimensional almost-minimal sets
  in $\mathbb{{R}}^n$.
\newblock {\em J. Geom. Anal.}, 20(4):837--954, 2010.

\bibitem{ddpt}
G.~David, T.~De~Pauw, and T.~Toro.
\newblock A generalization of {R}eifenberg's theorem in {$\Bbb R^3$}.
\newblock {\em Geom. Funct. Anal.}, 18(4):1168--1235, 2008.

\bibitem{davies2}
E.~B. Davies.
\newblock Properties of the {G}reen's functions of some {S}chr\"odinger
  operators.
\newblock {\em J. London Math. Soc. (2)}, 7:483--491, 1974.

\bibitem{lihewang}
G.~Hong and L.~Wang.
\newblock A geometric approach to the topological disk theorem of {R}eifenberg.
\newblock {\em Pacific J. Math.}, 233(2):321--339, 2007.

\bibitem{lihewang2}
G.~Hong and L.~Wang.
\newblock A new proof of {R}eifenberg's topological disc theorem.
\newblock {\em Pacific J. Math.}, 246(2):325--332, 2010.

\bibitem{jk}
D.~Jerison and C.~Kenig.
\newblock Boundary behavior of harmonic functions in nontangentially accessible
  domains.
\newblock {\em Adv. in Math.}, 46(1):80--147, 1982.

\bibitem{J}
P.~Jones.
\newblock Quasiconformal mappings and extendability of functions in {S}obolev
  spaces.
\newblock {\em Acta Math.}, 147(1-2):71--88, 1981.

\bibitem{hm3}
C.~Kenig and T.~Toro.
\newblock Harmonic measure on locally flat domains.
\newblock {\em Duke Math. J.}, 87(3):509--551, 1997.

\bibitem{hm2}
C.~Kenig and T.~Toro.
\newblock Free boundary regularity for harmonic measures and {P}oisson kernels.
\newblock {\em Ann. of Math. (2)}, 150(2):369--454, 1999.

\bibitem{hm1}
C.~Kenig and T.~Toro.
\newblock Poisson kernel characterization of {R}eifenberg flat chord arc
  domains.
\newblock {\em Ann. Sci. \'Ecole Norm. Sup. (4)}, 36(3):323--401, 2003.

\bibitem{l1}
A.~Lemenant.
\newblock Energy improvement for energy minimizing functions in the complement
  of generalized {R}eifenberg-flat sets.
\newblock {\em Ann. Scu. Norm. Sup. Pisa}, IX (5) 2010:1--34, 2010.

\bibitem{l2}
A.~Lemenant.
\newblock Regularity of the singular set for {M}umford-{S}hah minimizers in
  {$\Bbb R^3$} near a minimal cone.
\newblock {\em Ann. Sc. Norm. Super. Pisa Cl. Sci. (5)}, 10(3):561--609, 2011.

\bibitem{lm2}
A.~Lemenant and E.~Milakis.
\newblock Quantitative stability for the first {D}irichlet eigenvalue in
  {R}eifenberg-flat domains in $\mathbb{R}^n$.
\newblock {\em J. Math. Anal. Appl.}, 364:522--533, 2010.

\bibitem{lm}
A.~Lemenant and E.~Milakis.
\newblock A stability result for {N}onlinear {N}eumann problems in {R}eifenberg
  flat domains in $\mathbb{R}^n$.
\newblock {\em Publ. Mat.}, 55(2):413--432, 2011.

\bibitem{lms2}
A.~Lemenant, E.~Milakis, and L.~V. Spinolo.
\newblock Spectral {S}tability {E}stimates for the {D}irichlet and {N}eumann
  {L}aplacian in rough domains.
\newblock {\em Preprint 2012}, available on http://cvgmt.sns.it/papers.

\bibitem{mt}
E.~Milakis and T.~Toro.
\newblock Divergence form operators in {R}eifenberg flat domains.
\newblock {\em Math. Z.}, 264(1):15--41, 2010.

\bibitem{r}
E.~R. Reifenberg.
\newblock Solution of the {P}lateau {P}roblem for {$m$}-dimensional surfaces of
  varying topological type.
\newblock {\em Acta Math.}, 104:1--92, 1960.

\bibitem{marty}
M.~Ross.
\newblock The {L}ipschitz continuity of {N}eumann eigenvalues on convex
  domains.
\newblock {\em Hokkaido Math. J.}, 33(2):369--381, 2004.

\bibitem{Stein}
E.~M. Stein.
\newblock Singular integrals and differentiability properties of functions.
\newblock {\em Princeton Mathematical Series, Princeton University Press,
  Princeton, N.J.}, IV(30), 1970.

\bibitem{t}
T.~Toro.
\newblock Doubling and flatness: geometry of measures.
\newblock {\em Notices Amer. Math. Soc.}, 44(9):1087--1094, 1997.

\bibitem{hm4}
T.~Toro.
\newblock Geometry of measures: harmonic analysis meets geometric measure
  theory.
\newblock In {\em Handbook of geometric analysis. {N}o. 1}, volume~7 of {\em
  Adv. Lect. Math. (ALM)}, pages 449--465. Int. Press, Somerville, MA, 2008.

\end{thebibliography}

\begin{comment}
\vspace{0.5cm}
\noindent Antoine Lemenant, LJLL, Universit\'e Paris 7 - CNRS, {\small \tt lemenant@ljll.univ-paris-diderot.fr}

\vspace{0.5cm}
\noindent Emmanouil Milakis, University of Cyprus, {\small \tt emilakis@ucy.ac.cy}

\vspace{0.5cm}
\noindent Laura V. Spinolo, IMATI-CNR, { \small \tt spinolo@imati.cnr.it}
\end{comment}

\vspace{0.3cm}
\begin{tabular}{l}
Antoine Lemenant\\
Universit\'e Paris Diderot - Paris 7 - LJLL - CNRS \\
U.F.R de Math\'ematiques \\
Site Chevaleret Case 7012\\
75205 Paris Cedex 13 FRANCE\\
{e-mail : \small \tt lemenant@ljll.univ-paris-diderot.fr}
\end{tabular}
\vspace{2em}

\begin{tabular}{l}
Emmanouil Milakis\\ 
University of Cyprus \\ 
Department of Mathematics \& Statistics \\ 
P.O. Box 20537\\
Nicosia, CY- 1678 CYPRUS\\
 {e-mail : \small \tt emilakis@ucy.ac.cy}
\end{tabular}
\vspace{2em}

\begin{tabular}{l}
Laura V. Spinolo\\
IMATI-CNR, \\
via Ferrata 1 \\
I-27100, Pavia, ITALY  \\
{e-mail : \small \tt spinolo@imati.cnr.it}
\hfill
\end{tabular}

\end{document}